\documentclass{amsart}
\usepackage{amssymb}
\usepackage{enumitem}
\setlist{itemsep=4pt, topsep=0pt, leftmargin=17pt}

\usepackage{mathtools}
\usepackage{subcaption}

\makeatletter
\renewcommand{\p@subfigure}{}
\makeatother
\usepackage{placeins}
\usepackage{xcolor}
\definecolor{PKU}{cmyk}{0, 1, 1, .45}
\definecolor{BIT}{cmyk}{1, 0, 1, 0}
\usepackage[colorlinks, citecolor=BIT, linkcolor=PKU]{hyperref}
\usepackage[capitalize]{cleveref}
\AddToHook{env/conjecture/begin}{\crefalias{theorem}{conjecture}}
\AddToHook{env/lemma/begin}{\crefalias{theorem}{lemma}}
\AddToHook{env/corollary/begin}{\crefalias{theorem}{corollary}}
\AddToHook{env/proposition/begin}{\crefalias{theorem}{proposition}}

\newtheorem{theorem}{Theorem}[section]
\newtheorem{conjecture}[theorem]{Conjecture}
\newtheorem{lemma}[theorem]{Lemma}
\newtheorem{corollary}[theorem]{Corollary}
\newtheorem{proposition}[theorem]{Proposition}
\numberwithin{equation}{section}

\allowdisplaybreaks

\usepackage[numbers, sort&compress, nonamebreak, merge, elide, longnamesfirst]{natbib}

\crefname{figure}{Figure}{Figures}
\Crefname{figure}{Figure}{Figures}

\usepackage{tikz}
\usetikzlibrary{decorations.pathreplacing, calc, positioning, arrows.meta}
\tikzset{
vertex/.style={shape=circle, minimum size=1mm, fill=black, inner sep=0pt},
edge/.style={black, very thick}
}

\theoremstyle{definition}

\theoremstyle{remark}

\theoremstyle{plain}

\newcommand{\Mcal}{\mathcal M}
\newcommand{\avm}{\operatorname{avm}}
\newcommand{\pc}{\operatorname{pm}}
\newcommand{\defi}{\operatorname{def}}
\newcommand{\Aclass}{\mathcal A}
\newcommand{\Bclass}{\mathcal B}
\newcommand{\Cclass}{\mathcal C}
\newcommand{\Th}{\Theta}
\newcommand{\Dumb}{\mathsf D}

\newcommand{\AU}{\mathcal U}
\newcommand{\AV}{\mathcal V}
\newcommand{\AW}{\mathcal W}
\newcommand{\BO}{\mathcal O^{B}}
\newcommand{\BE}{\mathcal E^{B}}
\newcommand{\CO}{\mathcal O^{C}}
\newcommand{\CP}{\mathcal P^{C}}
\newcommand{\CQ}{\mathcal Q^{C}}
\newcommand{\CR}{\mathcal R^{C}}

\author[K. Zhang]{K. Zhang}
\address[K. Zhang]{School of Mathematics and Statistics, Beijing Institute of Technology, Beijing 102400, P.\ R.\ China.}
\email{kai@bit.edu.cn}


\title{Upper bounds for the average size of maximal matchings in bicyclic graphs}

\begin{document}

\begin{abstract}
For a graph $G$, let avm($G$) denote the average size of its maximal matchings. Engbers and Erey initiated the extremal study of this parameter and asked for extensions from trees and unicyclic graphs to $k$-cyclic graphs. In this paper, we determine the maximum value of avm($G$) over all connected bicyclic graphs with $n$ vertices and $n+1$ edges. If $n\ge 5$ is odd, then
\[
\text{avm}(G)\le \frac{n-1}{2},
\]
and we characterize all graphs attaining equality. For $n=6$, the maximum value is 13/5, attained uniquely by $\Theta(1,3,3)$. If $n\ge 8$ is even, then
\[
\text{avm}(G)\le \frac{n}{2}-1+\frac{2}{n-4}.
\]
Equality holds precisely for the graph obtained from two copies of $C_4$ joined by an edge by attaching $(n-8)/2$ pendant 2-paths to one endpoint of the joining edge, and, when $n\ge 10$, for the graph obtained from two copies of $C_4$ joined by a path of length 2 by attaching one leaf and $(n-10)/2$ pendant 2-paths to the internal vertex of the joining path. The proofs combine structural characterizations of odd-order extremal graphs with counting and switching arguments based on perfect matchings.
\end{abstract}

\maketitle
\tableofcontents

\section{Introduction}\label{sec:intro}

Throughout this paper, all graphs are finite, simple, and undirected. For a graph $G$, its vertex set and edge set are denoted by $V(G)$ and $E(G)$, respectively; its order is $|V(G)|$ and its size is $|E(G)|$. A connected graph of order $n$ is called $c$-cyclic if it has $n-1+c$ edges. Thus trees, unicyclic graphs, and bicyclic graphs correspond to $c=0,1$, and $2$, respectively.

A \emph{matching} of $G$ is a set of pairwise nonincident edges. A matching $M$ \emph{covers} a vertex if that vertex is incident with an edge of $M$. It is \emph{maximal} if it is not properly contained in another matching, \emph{maximum} if it has largest possible size, and \emph{perfect} if it covers every vertex. The size of a maximum matching is the matching number $\nu(G)$, while the minimum size of a maximal matching is the saturation number $s(G)$. Hence every maximal matching has size between $s(G)$ and $\nu(G)$. Maximal matchings are precisely the independent edge-dominating sets: their edges are pairwise nonincident, and every edge outside the matching is incident with an edge of the matching.

Average graph parameters refine extremal parameters by recording the distribution of all feasible configurations rather than only the smallest or largest one. For ordinary matchings, Andriantiana, Misanantenaina, and Wagner~\cite{AndriantianaMatchings} studied the average size and determined extremal trees. Engbers and Erey~\cite{EngbersErey2023} initiated the corresponding extremal study for maximal matchings. Let $\Mcal(G)$ be the family of all maximal matchings of $G$, and let $m_i(G)$ be the number of its maximal matchings of size $i$. Set
\[
 m(G)=|\Mcal(G)|=\sum_i m_i(G)
 \quad\text{and}\quad
 m'(G)=\sum_{M\in\Mcal(G)}|M|=\sum_i i\,m_i(G).
\]
The \emph{average size of a maximal matching} of $G$ is
\[
 \avm(G)=\frac{1}{|\Mcal(G)|}\sum_{M\in\Mcal(G)}|M|
        =\frac{m'(G)}{m(G)}.
\]
Equivalently, $\avm(G)$ is the expected size of a maximal matching chosen uniformly from $\Mcal(G)$. Although
\[
 s(G)\le \avm(G)\le \nu(G),
\]
the value of $\avm(G)$ is not determined by the two endpoint parameters: it depends on the entire size distribution $(m_i(G))_i$. This makes its extremal behavior simultaneously structural and enumerative.

The basic extremal problem posed by Engbers and Erey can be stated as follows.

\medskip
\noindent\textbf{Question 1.} Which graphs $G$ have the maximum or minimum value of $\avm(G)$ when $G$ is restricted to a particular family?
\medskip

Engbers and Erey~\cite{EngbersErey2023} answered this question for trees, connected unicyclic graphs, trees of fixed order and diameter at most $5$, and $2$-regular graphs. Their work also led to more refined problems. For example, they proposed the following conjecture for the minimum among trees of fixed order and diameter.

\begin{conjecture}\label{conj:diameter}
Let $T$ be a tree of order $n$ with diameter $d$, where $d\neq4$ and $d\ge3$. Then $\avm(T)$ is uniquely minimized by the tree obtained from a path of length $d-1$ by attaching $n-d$ leaves to one endpoint of the path.
\end{conjecture}

Sui and Li~\cite{Sui2025} subsequently verified \cref{conj:diameter} for $d=6,7$, and $8$. In a different direction, Engbers and Erey explicitly asked for extensions beyond the unicyclic case.

\medskip
\noindent\textbf{Question 2.} Extend their results on $\avm(G)$ from unicyclic graphs to $k$-cyclic graphs with $k\ge 2$.
\medskip

The first case not covered by the tree and unicyclic results is therefore the family of connected bicyclic graphs. In our companion paper~\cite{ZhangLower}, we determined the minimum of $\avm(G)$ over this family. The purpose of the present paper is to complete the two-sided extremal picture by determining the maximum and characterizing all graphs attaining the maximum. The upper-bound problem has a different character from the lower-bound problem. When $n$ is odd, the numerical inequality $\avm(G)\le(n-1)/2$ is immediate, but equality requires a complete classification of the bicyclic graphs in which every maximal matching is maximum. When $n$ is even, the average can exceed $n/2-1$ only through perfect matchings; the proof must therefore control both the number of perfect matchings and the maximal matchings of size $n/2-1$ produced from them.

Let $\mathcal G(n,n+1)$ denote the family of connected simple graphs with $n$ vertices and $n+1$ edges. Following the classification and notation used in~\cite{Deng2008,ZhangLower}, every graph in $\mathcal G(n,n+1)$ belongs to one of three classes:
\begin{itemize}
\item $\Aclass(p,q,l)$, when the two cycles $C_p$ and $C_q$ share a path of length $l$;
\item $\Bclass(p,q)$, when the two cycles have exactly one common vertex;
\item $\Cclass(p,q)$, when the two cycles are vertex-disjoint and are joined by a path.
\end{itemize}
For a graph in $\Aclass(p,q,l)$ or $\Bclass(p,q)$, the \emph{core} is the subgraph induced by the vertices on the two cycles. For a graph in $\Cclass(p,q)$, the core consists of the two cycles together with the path joining them. The three core types are shown in \cref{core}. Every connected bicyclic graph is obtained from one of these cores by attaching rooted trees to core vertices.

\begin{figure}
\centering
\subcaptionbox{The core of $\mathcal{A}(p, q, l)$\label{apql}}[0.28\textwidth]{
\begin{tikzpicture}[scale=0.85]
\node[vertex] (u) at (-1.5,0) {};
\node[vertex] (q) at (-1.0,0) {};
\node[vertex] (w) at (1.0,0) {};
\node[vertex] (v) at (1.5,0) {};
\draw[edge] (u) -- (q);
\draw[edge] (w) -- (v);
\node[vertex] (x) at (-0.9,0.7) {};
\node[vertex] (y) at (0.9,0.7) {};
\node[vertex] (e) at (0.9,-0.7) {};
\node[vertex] (r) at (-0.9,-0.7) {};
\draw[edge] (u) -- (x);
\draw[edge] (u) -- (r);
\draw[edge] (v) -- (e);
\draw[edge] (v) -- (y);
\draw[thick] (-0.9,0.7)--(-0.5,0.7);
\draw[thick] (0.9,0.7)--(0.5,0.7);
\draw[thick] (-1,0)--(-0.6,0);
\draw[thick] (1,0)--(0.6,0);
\draw[thick] (-0.9,-0.7)--(-0.5,-0.7);
\draw[thick] (0.9,-0.7)--(0.5,-0.7);
\node at (0,0) {$\cdots$};
\node at (0,0.7) {$\cdots$};
\node at (0,-0.8) {$\cdots$};
\end{tikzpicture}
}
\hspace{1em}
\subcaptionbox{The core of $\mathcal{B}(p, q)$ \label{bpq}}[0.28\textwidth]{
\begin{tikzpicture}[scale=0.85]
\node[vertex] (c) at (0,0) {};
\node[vertex] (l1) at (-1.3, -1) {};
\node[vertex] (l2) at (1.3, -1) {};
\node[vertex] (l3) at (-0.8, -1) {};
\node[vertex] (l4) at (0.8, -1) {};
\draw[thick] (-0.8, -1) -- (-0.4, -1);
\draw[thick] (0.8, -1) -- (0.4, -1);
\draw[edge] (c) -- (l1) -- (l3);
\draw[edge] (c) -- (l2) -- (l4);
\node[vertex] (r1) at (-1.3, 1) {};
\node[vertex] (r2) at (1.3, 1) {};
\node[vertex] (r3) at (-0.8, 1) {};
\node[vertex] (r4) at (0.8, 1) {};
\draw[thick] (-0.8, 1) -- (-0.4, 1);
\draw[thick] (0.8, 1) -- (0.4, 1);
\draw[edge] (c) -- (r1) -- (r3);
\draw[edge] (c) -- (r2) -- (r4);
\node at (0,1) {$\cdots$};
\node at (0,-1) {$\cdots$};
\end{tikzpicture}
}
\hspace{1em}
\subcaptionbox{The core of $\mathcal{C}(p, q)$\label{cpq}}[0.28\textwidth]{
\begin{tikzpicture}[scale=0.85]
\node[vertex] (v3) at (-0.7,0)  {};
\node[vertex] (v6) at (-0.2,0)  {};
\node[vertex] (v7) at (1.2,0)  {};
\draw[thick] (v3)--(0.1,0);
\node[vertex] (v1) at (-1.5, 1)  {};
\node[vertex] (v2) at (-1.5, -1) {};
\node[vertex] (v4) at (-1.5, 0.6)  {};
\node[vertex] (v5) at (-1.5, -0.6) {};
\draw[thick] (-1.5, 0.6)--(-1.5, 0.3);
\draw[thick] (-1.5, -0.6)--(-1.5, -0.3);
\node[vertex] (w1) at (1.7,0)  {};
\node[vertex] (w2) at (2.5, 1) {};
\node[vertex] (w3) at (2.5, -1) {};
\node[vertex] (w4) at (2.5, 0.6) {};
\node[vertex] (w5) at (2.5, -0.6){};
\draw[thick] (2.5, 0.6)--(2.5, 0.3);
\draw[thick] (2.5, -0.6)--(2.5, -0.3);
\draw[thick] (0.9,0)--(w1);
\draw[edge] (v3) -- (v1) -- (v4);
\draw[edge] (v3) -- (v2) -- (v5);
\draw[edge] (w1) -- (w2) -- (w4);
\draw[edge] (w1) -- (w3) -- (w5);
\node at (0.5,0) {$\cdots$};
\node at (-1.5,0.1) {$\vdots$};
\node at (2.5,0.1) {$\vdots$};
\end{tikzpicture}
}
\caption{The three types of bicyclic cores.}\label{core}
\end{figure}
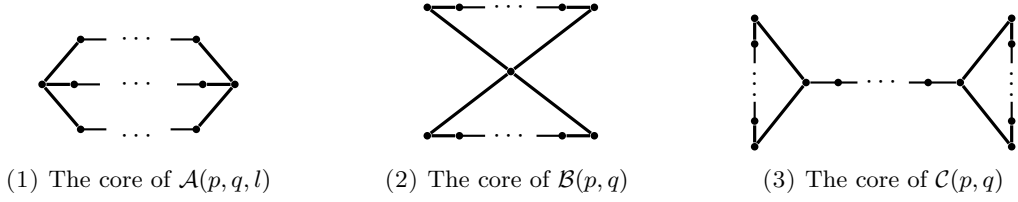

This tripartite decomposition is essential because the three cores support different numbers of perfect matchings. A graph with a theta core has at most three perfect matchings, a graph whose cycles meet in one vertex has at most two, and a graph with two disjoint cycles has at most four. Consequently, the extremal values and equality structures must first be determined separately in the three classes and then compared.

Our main result is the following.

\begin{theorem}\label{thm:global}
Let $G\in\mathcal G(n,n+1)$, where $n\ge5$.
\begin{enumerate}[label=\textup{(\roman*)}]
\item If $n$ is odd, then
\[
 \avm(G)\le\frac{n-1}{2}.
\]
The graphs attaining equality are exactly those described in \cref{thm:A-odd,thm:B-odd,thm:C-odd}.

\item If $n=6$, then
\[
 \avm(G)\le\frac{13}{5},
\]
with equality if and only if $G\cong\Th(1,3,3)$.

\item If $n\ge8$ is even, then
\[
 \avm(G)\le\frac n2-1+\frac{2}{n-4}.
\]
Equality holds precisely for the following graphs: two copies of $C_4$ joined by one edge, with $(n-8)/2$ pendant $2$-paths attached to one endpoint of the joining edge; or, when $n\ge10$, two copies of $C_4$ joined by a path of length $2$, with one leaf and $(n-10)/2$ pendant $2$-paths attached to the internal vertex of the joining path.
\end{enumerate}
\end{theorem}

The restriction $n\ge5$ agrees with the range in the lower-bound paper. For completeness, the only connected simple bicyclic graph of order $4$ is $K_4-e\cong\Th(1,2,2)$, and $\avm(K_4-e)=5/3$.

The remainder of the paper is organized as follows. Section~\ref{sec:prelim} collects notation and the tools used in more than one core class. Sections~\ref{sec:A}, \ref{sec:B}, and \ref{sec:C} determine the sharp upper bounds and equality cases for $\Aclass(p,q,l)$, $\Bclass(p,q)$, and $\Cclass(p,q)$, respectively. Section~\ref{sec:global} compares the three classwise bounds and completes the proof of \cref{thm:global}.

\section{Preliminaries}\label{sec:prelim}

\subsection{Notation and bicyclic cores}

For a vertex $v$ of a graph $G$, let $N_G(v)$ denote its neighborhood. For $S\subseteq V(G)$, the graph $G-S$ is the subgraph induced by $V(G)\setminus S$. We write $P_t$ and $C_t$ for the path and cycle of order $t$, respectively. A vertex of degree $1$ is a \emph{leaf}; a vertex adjacent to a leaf is a \emph{support vertex}; and an edge incident with a leaf is a \emph{pendant edge}. If a support vertex has degree $2$, its nonpendant incident edge is called its \emph{support edge}. We denote by $\pc(G)$ the number of perfect matchings of $G$.

A graph is \emph{equimatchable} if all its maximal matchings have the same size. In particular, an odd-order graph $G$ satisfies $\avm(G)=(|V(G)|-1)/2$ exactly when every maximal matching is maximum. A graph is \emph{randomly matchable} if every maximal matching is perfect.

We use the following core notation throughout. For positive integers $a,b,c$, with at most one of them equal to $1$, let $\Th(a,b,c)$ be the union of three internally disjoint paths of lengths $a,b,c$ with common endvertices. The notation is invariant under permutations of $a,b,c$. Thus $|V(\Th(a,b,c))|=a+b+c-1$, and the core of a graph in $\Aclass(p,q,l)$ is
\[
 \Th(l,p-l,q-l).
\]
Let $F_{p,q}$ be the union of $C_p$ and $C_q$ with exactly one common vertex; this is the core of a graph in $\Bclass(p,q)$. Finally, let $\Dumb_l(p,q)$ consist of two vertex-disjoint cycles $C_p$ and $C_q$ whose distinguished vertices are joined by a path of length $l$; this is the core of a graph in $\Cclass(p,q)$. We write $\Aclass_n$, $\Bclass_n$, and $\Cclass_n$ for the corresponding subfamilies of $\mathcal G(n,n+1)$.

For the purposes of this paper, a bicyclic graph is called \emph{reduced} if no support vertex outside its core has degree $2$. This terminology is used for the graphs at which the repeated deletion of a degree-$2$ support vertex outside the core together with its leaf terminates.

A \emph{pendant $2$-path attached at $w$} consists of two new vertices $s,t$ and the two new edges $ws,st$. We repeatedly use the elementary fact that a matching is maximal if and only if its uncovered vertices form an independent set.

\subsection{Pendant $2$-paths and structural reduction}

\begin{lemma}\label{lem:support}
Every maximal matching covers every support vertex. If a graph has a perfect matching, then every support vertex has exactly one leaf neighbor.
\end{lemma}

\begin{proof}
If a support vertex $s$ and one of its leaf neighbors are both uncovered, their pendant edge can be added. Thus $s$ is covered by every maximal matching. If $s$ had two leaf neighbors, no matching could cover both leaves, so no perfect matching would exist.
\end{proof}

\begin{lemma}\label{lem:odd-support}
Let $G$ be a connected bicyclic graph of order $2k+1$ in which every maximal matching has size $k$. Then every support vertex of $G$ has exactly one leaf neighbor.
\end{lemma}

\begin{proof}
Suppose that a support vertex $s$ has at least two leaf neighbors. Since a bicyclic graph is not a star, $s$ has a nonleaf neighbor $x$. Extend the edge $sx$ to a maximal matching $M$ of $G$. None of the leaf neighbors of $s$ is covered by $M$, so at least two vertices are uncovered. Because $|V(G)|$ is odd, the number of uncovered vertices is odd and hence at least three. Therefore $|M|\le k-1$, a contradiction.
\end{proof}

\begin{lemma}\label{lem:P2}
Let $H$ be a graph and $w\in V(H)$. Form $G$ by adding new vertices $s,t$ and the edges $ws,st$. Then every maximal matching of $G$ is uniquely of one of the forms
\[
 M\cup\{st\},\qquad M\in\Mcal(H),
\]
or
\[
 N\cup\{ws\},\qquad N\in\Mcal(H-w).
\]
Consequently,
\[
 m(G)=m(H)+m(H-w)
\]
and
\[
 m'(G)=m'(H)+m(H)+m'(H-w)+m(H-w).
\]
Moreover, $G$ and $H$ have the same number of perfect matchings.
\end{lemma}

\begin{proof}
Every maximal matching covers the support vertex $s$. If it uses $st$, its restriction to $H$ is maximal. If it does not use $st$, it must use $ws$, and its remaining edges form a maximal matching of $H-w$. The two cases are disjoint and exhaustive. A perfect matching must use $st$, proving the final assertion.
\end{proof}

\begin{corollary}\label{cor:odd-extension}
Suppose $H$ has $2r+1$ vertices and every maximal matching of $H$ has size $r$. Let $G$ be obtained by attaching a pendant $2$-path at $w\in V(H)$. Then every maximal matching of $G$ has size $r+1$ if and only if every maximal matching of $H-w$ is perfect.
\end{corollary}

\begin{proof}
Apply \cref{lem:P2}. The first class always has size $r+1$. The second class has size $1+|N|$, so it has size $r+1$ for every $N$ exactly when every maximal matching of the even-order graph $H-w$ has size $r$.
\end{proof}

\begin{lemma}\label{lem:farthest}
Let $G$ be obtained by attaching trees to a fixed bicyclic core, and suppose every support vertex has one leaf neighbor. If a support vertex lies outside the core, then some support vertex outside the core has degree $2$.
\end{lemma}

\begin{proof}
Choose a support vertex farthest from the core. Apart from its unique neighbor toward the core, all its neighbors lie farther from the core and hence are leaves. Since it has only one leaf neighbor, its degree is $2$.
\end{proof}

\subsection{Randomly matchable graphs and cut vertices}

We need only the sparse case of the classical characterization of randomly matchable graphs.

\begin{lemma}\label{lem:sparse-random}
Let $H$ be a graph each of whose connected components has cyclomatic number at most one. Then every maximal matching of $H$ is perfect if and only if every connected component of $H$ is isomorphic to $K_2$ or $C_4$.
\end{lemma}

\begin{proof}
By the characterization of Sumner~\cite{Sumner}, a connected randomly matchable graph is either a complete graph of even order or a complete bipartite graph $K_{r,r}$. Under cyclomatic number at most one, the only possibilities are $K_2$ and $C_4$. Conversely, every maximal matching in $K_2$ or $C_4$ is perfect, and this property is preserved under disjoint union.
\end{proof}

\begin{lemma}\label{lem:tree-two-deletions}
Let $T$ be a tree and let $a,b\in V(T)$ be distinct. If both $T-a$ and $T-b$ are disjoint unions of copies of $K_2$, then $T\cong P_3$, with $a$ and $b$ as its endvertices.
\end{lemma}

\begin{proof}
For every $v\in V(T)\setminus\{a,b\}$, we have $d_{T-a}(v)=d_{T-b}(v)=1$. Hence $v$ is adjacent to $a$ if and only if it is adjacent to $b$. If $v$ were adjacent to neither, then $d_T(v)=1$; its unique neighbor could not be adjacent to $a$ or $b$, and the resulting copy of $K_2$ would be a component of $T$, contrary to connectedness. Thus every vertex outside $\{a,b\}$ is adjacent to both $a$ and $b$. Since $T$ is a tree, there is exactly one such vertex, and $T$ is the path $a v b$.
\end{proof}

\begin{lemma}\label{lem:cut-splice}
Let $x$ be a cut vertex of $G$, and write
\[
 G-x=H_1\cup\cdots\cup H_s.
\]
For each $i$, let $S_i=N_G(x)\cap V(H_i)$. Fix $j$ and $z\in S_j$. If $M_j$ is a maximal matching of $H_j-z$ and $M_i$ is a maximal matching of $H_i$ for $i\ne j$, then
\[
 \{xz\}\cup M_j\cup\bigcup_{i\ne j}M_i
\]
is a maximal matching of $G$.
\end{lemma}

\begin{proof}
The edge $xz$ dominates all edges incident with $x$ or $z$, and each $M_i$ dominates the edges internal to its component. Hence every edge of $G$ is dominated by the displayed matching.
\end{proof}

\subsection{Perfect-matching switching}

Throughout the switching arguments, $G$ has even order $2k$ and size $2k+1$.

\begin{lemma}\label{lem:switch}
Let $M$ be a perfect matching of $G$, let $e=xy\notin M$, and let $xx',yy'\in M$. Set
\[
 M_e=(M\setminus\{xx',yy'\})\cup\{xy\}.
\]
Then $M_e$ has size $k-1$ and leaves precisely $x',y'$ uncovered. It is maximal if and only if $x'y'\notin E(G)$. If $x'y'\in E(G)$, then $M_e\cup\{x'y'\}$ is a second perfect matching.
\end{lemma}

\begin{proof}
All vertices except $x',y'$ remain covered. Therefore the only edge that can be added is $x'y'$.
\end{proof}

For a perfect matching $M$, let
\[
 \mathcal S(M)=\{M_e:e\in E(G)\setminus M\text{ and }M_e\text{ is maximal}\}
\]
be its \emph{switching family}.

\begin{lemma}\label{lem:switch-distance}
Let $M$ and $M'$ be perfect matchings. If $Q\in\mathcal S(M)\cap\mathcal S(M')$, then
\[
 |M\triangle M'|\le6.
\]
Moreover, the following statements hold.
\begin{enumerate}[label=\textup{(\roman*)}]
\item If $|M\triangle M'|=6$, then $|\mathcal S(M)\cap\mathcal S(M')|\le3$.
\item If $|M\triangle M'|=4$, then $M\triangle M'$ is an alternating $4$-cycle. Every common switched matching is obtained by switching along a diagonal of this cycle. Consequently, the two switching families have at most one common member; if the alternating $4$-cycle is chordless, they are disjoint.
\end{enumerate}
\end{lemma}

\begin{proof}
A switch changes exactly three edges, so
\[
 |M\triangle M'|\le |M\triangle Q|+|Q\triangle M'|=6.
\]
The symmetric difference of two perfect matchings is a disjoint union of even alternating cycles. If its size is $6$, it is one alternating $6$-cycle. A common switched matching retains one edge from each of the two alternating perfect matchings on this cycle, and the retained edges must be nonincident. There are three possible pairs.

If the symmetric difference has size $4$, it is one alternating $4$-cycle. Since $M$ and $M'$ agree outside this cycle, a direct comparison of the two switches shows that a common switched matching contains all their common edges and one diagonal of the $4$-cycle. If the cycle has no diagonal, no common member exists. If it has exactly one diagonal, that diagonal gives at most one common maximal switch. If it has both diagonals, either diagonal leaves the endpoints of the other diagonal uncovered, so neither switch is maximal.
\end{proof}

\section{Upper bounds in the class \texorpdfstring{$\mathcal{A}(p,q,l)$}{A(p,q,l)}}\label{sec:A}

Let $G\in\Aclass_n$. We shall use three parameterized graph families. For $r\ge0$, let $\AU_r$ be obtained from $\Th(1,2,3)$ by attaching $r$ pendant $2$-paths at the internal vertex of its length-$2$ path; let $\AV_r$ be obtained from $\Th(2,2,2)\cong K_{2,3}$ by attaching $r$ pendant $2$-paths at one degree-$2$ vertex; and let $\AW_r$ be obtained from $\Th(1,3,3)$ by attaching $r$ pendant $2$-paths at one of its two branch vertices. Their orders are $5+2r$, $5+2r$, and $6+2r$, respectively; see Figure~\ref{fig:A-extremal}(\ref{fig:A-U})--(\ref{fig:A-W}). The two additional graphs that occur in the odd-order equality statement are the graph obtained from $\Th(1,2,2)$ by attaching one leaf to an internal vertex of a length-$2$ path and the graph $\Th(1,3,4)$; see Figure~\ref{fig:A-extremal}(\ref{fig:A-exception-five}) and Figure~\ref{fig:A-extremal}(\ref{fig:A-exception-seven}).

\begin{figure}[htbp]
\centering
\subcaptionbox{$\AU_r\label{fig:A-U}$}[0.31\textwidth]{%
\begin{tikzpicture}[scale=0.85]
\node[vertex] (u) at (-1.45,0) {};
\node[vertex] (v) at (1.45,0) {};
\node[vertex] (a) at (0,0.85) {};
\node[vertex] (b) at (-0.35,-0.8) {};
\node[vertex] (c) at (0.7,-0.8) {};
\draw[edge] (u)--(v);
\draw[edge] (u)--(a)--(v);
\draw[edge] (u)--(b)--(c)--(v);
\node[vertex] (s1) at (-0.45,1.55) {};
\node[vertex] (t1) at (-0.8,2.15) {};
\node[vertex] (s2) at (0.45,1.55) {};
\node[vertex] (t2) at (0.8,2.15) {};
\draw[edge] (a)--(s1)--(t1);
\draw[edge] (a)--(s2)--(t2);
\node at (0,1.65) {$\cdots$};
\end{tikzpicture}}
\hfill
\subcaptionbox{$\AV_r\label{fig:A-V}$}[0.31\textwidth]{%
\begin{tikzpicture}[scale=0.85]
\node[vertex] (u) at (-1.3,0) {};
\node[vertex] (v) at (1.3,0) {};
\node[vertex] (a) at (0,0.95) {};
\node[vertex] (b) at (0,0) {};
\node[vertex] (c) at (0,-0.95) {};
\draw[edge] (u)--(a)--(v);
\draw[edge] (u)--(b)--(v);
\draw[edge] (u)--(c)--(v);
\node[vertex] (s1) at (-0.48,1.70) {};
\node[vertex] (t1) at (-0.88,2.32) {};
\node[vertex] (s2) at (0.48,1.70) {};
\node[vertex] (t2) at (0.88,2.32) {};
\draw[edge] (a)--(s1)--(t1);
\draw[edge] (a)--(s2)--(t2);
\node at (0,1.80) {$\cdots$};
\end{tikzpicture}}
\hfill
\subcaptionbox{$\AW_r\label{fig:A-W}$}[0.31\textwidth]{%
\begin{tikzpicture}[scale=0.83]
\node[vertex] (u) at (-1.2,0) {};
\node[vertex] (v) at (1.4,0) {};
\node[vertex] (a) at (-0.45,0.9) {};
\node[vertex] (b) at (0.45,0.9) {};
\node[vertex] (c) at (-0.45,-0.9) {};
\node[vertex] (d) at (0.45,-0.9) {};
\draw[edge] (u)--(v);
\draw[edge] (u)--(a)--(b)--(v);
\draw[edge] (u)--(c)--(d)--(v);
\node[vertex] (s1) at (-1.9,0.8) {};
\node[vertex] (t1) at (-2.45,1.35) {};
\node[vertex] (s2) at (-1.9,-0.8) {};
\node[vertex] (t2) at (-2.45,-1.35) {};
\draw[edge] (u)--(s1)--(t1);
\draw[edge] (u)--(s2)--(t2);
\node at (-1.95,0) {$\vdots$};
\end{tikzpicture}}

\par\medskip
\subcaptionbox{The graph obtained from $\Th(1,2,2)$ by attaching one leaf to an internal vertex of a length-$2$ path\label{fig:A-exception-five}}[0.47\textwidth]{%
\begin{tikzpicture}[scale=0.88]
\node[vertex] (u) at (-1.3,0) {};
\node[vertex] (v) at (1.3,0) {};
\node[vertex] (a) at (-0.35,0.85) {};
\node[vertex] (b) at (-0.35,-0.85) {};
\node[vertex] (y) at (-0.35,1.7) {};
\draw[edge] (u)--(v);
\draw[edge] (u)--(a)--(v);
\draw[edge] (u)--(b)--(v);
\draw[edge] (a)--(y);
\end{tikzpicture}}
\hfill
\subcaptionbox{$\Th(1,3,4)\label{fig:A-exception-seven}$}[0.31\textwidth]{%
\begin{tikzpicture}[scale=0.83]
\node[vertex] (u) at (-1.35,0) {};
\node[vertex] (v) at (1.5,0) {};
\node[vertex] (a) at (-0.65,0.9) {};
\node[vertex] (b) at (0.25,0.9) {};
\node[vertex] (c) at (-0.7,-0.9) {};
\node[vertex] (d) at (0.15,-0.9) {};
\node[vertex] (e) at (0.85,-0.9) {};
\draw[edge] (u)--(v);
\draw[edge] (u)--(a)--(b)--(v);
\draw[edge] (u)--(c)--(d)--(e)--(v);
\end{tikzpicture}}
\caption{The graphs $\AU_r$, $\AV_r$, $\AW_r$, the exceptional graph of order $5$, and $\Th(1,3,4)$. In each parameterized graph, the fan with an ellipsis represents $r$ pendant $2$-paths; it is absent when $r=0$.}
\label{fig:A-extremal}
\end{figure}
\FloatBarrier

\subsection{Odd order}

For odd $n=2k+1$, every matching has size at most $k$, so $\avm(G)\le k$. Equality holds if and only if every maximal matching has size $k$. We classify these equimatchable graphs by repeatedly deleting pendant $2$-paths.

\begin{lemma}\label{lem:A-independent-three}
Let $T=\Th(a,b,c)$ have odd order at least $9$. Then there is an independent set $S$ of three vertices such that $T-S$ has a perfect matching. Hence $T$ has a maximal matching of size $(|V(T)|-3)/2$.
\end{lemma}

\begin{proof}
Write the three $u$--$v$ paths as $P_a,P_b,P_c$. Since $a+b+c$ is even, either all three lengths are even or exactly two are odd.

If all three are even and two, say $b,c$, are at least $4$, delete $u$ and the vertices immediately preceding $v$ on $P_b,P_c$. The three deleted vertices are independent, and the remaining pieces are even-order paths. If two paths have length $2$, say $a=b=2$, then $c\ge6$; delete the two internal vertices of $P_a,P_b$ and the second internal vertex of $P_c$. The remaining graph is a disjoint union of even paths.

Now suppose exactly two path lengths are odd. If both odd lengths $r,s$ are at least $3$ and the even length is $e$, delete the first internal vertex of $P_r$, the last internal vertex of $P_s$, and the first internal vertex of $P_e$. The residual path pieces pair through $v$ and have even order. If one odd path has length $1$, let the other odd length be $r$ and the even length be $e$. When $r\ge5$ and $e\ge4$, delete $u$, the third internal vertex of $P_r$, and the last internal vertex of $P_e$. When $r=3$, necessarily $e\ge6$; delete the first internal vertex of $P_r$ and the second and last internal vertices of $P_e$. When $e=2$, necessarily $r\ge7$; delete $u$ and the third and last internal vertices of $P_r$. In each subcase the deleted vertices are independent and all remaining components are even paths.

A perfect matching of $T-S$ is maximal in $T$ because the uncovered set $S$ is independent.
\end{proof}

\begin{lemma}\label{lem:A-pure-odd}
Let $T=\Th(a,b,c)$ have odd order. Every maximal matching of $T$ is maximum if and only if
\[
 (a,b,c)\in\{(1,2,3),(2,2,2),(1,3,4)\}.
\]
The corresponding maximal-matching generating polynomials are $6z^2,6z^2$, and $9z^3$.
\end{lemma}

\begin{proof}
By \cref{lem:A-independent-three}, the order is at most $7$. For order $5$, $a+b+c=6$, and the only admissible triples are $(1,2,3)$ and $(2,2,2)$. Directly, $\Th(1,2,3)$ has six maximal matchings, all of size $2$, and $\Th(2,2,2)=K_{2,3}$ has the six matchings $\{ux_i,vx_j\}$ with $i\ne j$.

For order $7$, the possible triples are
\[
 (1,2,5),\ (1,3,4),\ (2,2,4),\ (2,3,3).
\]
In each of the first, third, and fourth graphs, one can delete three independent vertices so that the residual graph has a perfect matching of size $2$; hence the graph has a nonmaximum maximal matching. For $\Th(1,3,4)$, write the nontrivial paths as $uabv$ and $ucdev$. Sorting by the edge that covers $u$ gives exactly nine maximal matchings:
\[
\begin{gathered}
\{uv,ab,cd\},\ \{uv,ab,de\},\
\{ua,bv,cd\},\ \{ua,bv,de\},\ \{ua,ev,cd\},\\
\{uc,bv,de\},\ \{uc,ev,ab\},\
\{uc,ab,de\},\ \{ev,ab,cd\}.
\end{gathered}
\]
All have size $3$.
\end{proof}

\begin{lemma}\label{lem:A-reduced-odd}
Let $G\in\Aclass_{2k+1}$ be reduced and suppose every maximal matching has size $k$. Then $G$ is isomorphic to one of $\Th(1,2,3)$, $\Th(2,2,2)$, $\Th(1,3,4)$, or the graph obtained from $\Th(1,2,2)$ by attaching one leaf to an internal vertex of a length-$2$ path.
\end{lemma}

\begin{proof}
By \cref{lem:odd-support}, every support vertex has exactly one leaf neighbor. If a leaf had a support vertex outside the core, then \cref{lem:farthest} would give a degree-$2$ support vertex outside the core, contrary to the definition of a reduced graph. Hence every leaf is attached directly to a core vertex. If there is no leaf, \cref{lem:A-pure-odd} applies.

Suppose that a leaf $y$ is attached to a core vertex $x$, and put $R=G-\{x,y\}$. For every core neighbor $z$ of $x$ and every maximal matching $N$ of $R-z$, the matching $N\cup\{xz\}$ is maximal in $G$. It therefore has size $k$, so $N$ is a perfect matching of the graph $R-z$, which has order $2k-2$. Thus $R-z$ is randomly matchable for every core neighbor $z$ of $x$. By \cref{lem:sparse-random}, every component of each such graph is $K_2$ or $C_4$.

First suppose that $x$ is a branch vertex of the theta core. Then $R$ is a tree and, for each of the three core neighbors $z$ of $x$, the graph $R-z$ is a disjoint union of copies of $K_2$. Applying \cref{lem:tree-two-deletions} to any two such neighbors would force $R\cong P_3$, which is incompatible with a third distinct neighbor having the same property. Hence $x$ is not a branch vertex.

Consequently, $x$ is internal to one theta path and has two core neighbors $z_1,z_2$. The graph $R$ is connected and unicyclic; its unique cycle is formed by the other two theta paths. If, say, $z_1$ is not on this cycle, then the component of $R-z_1$ containing the cycle must be isomorphic to $C_4$. This forces $z_2$ to lie on the cycle. But then $R-z_2$ contains the path obtained from that $C_4$ by deleting $z_2$, together with the nontrivial segment ending at $z_1$, and hence has a component that is neither $K_2$ nor $C_4$, a contradiction. Thus both $z_1$ and $z_2$ lie on the unique cycle, so the theta path containing $x$ has length $2$.

There can be no further leaf in $R$. Indeed, if its support vertex were $z_i$, then $R-z_i$ would contain an isolated vertex. Otherwise its pendant edge would survive in both $R-z_1$ and $R-z_2$ inside a component that is neither $K_2$ nor $C_4$. Both alternatives contradict \cref{lem:sparse-random}. Hence $R$ is a cycle. Since $R-z_i$ is then a path and is randomly matchable, it must be $K_2$. Therefore $R=C_3$, and $G$ is exactly the graph obtained from $\Th(1,2,2)$ by attaching one leaf to an internal vertex of a length-$2$ path.
\end{proof}

\begin{lemma}\label{lem:A-odd-vertices}
The vertices $w$ for which deleting $w$ from a reduced odd-order extremal graph is randomly matchable are as follows:
\begin{enumerate}[label=\textup{(\roman*)}]
\item in $\Th(1,2,3)$, only the internal vertex of the length-$2$ path;
\item in $\Th(2,2,2)$, precisely the three degree-$2$ vertices;
\item in $\Th(1,3,4)$ and in the graph obtained from $\Th(1,2,2)$ by attaching one leaf to an internal vertex of a length-$2$ path, no vertex.
\end{enumerate}
After at least one pendant $2$-path has been attached in the first two cases, the common attachment vertex is the unique vertex with this property.
\end{lemma}

\begin{proof}
Deleting the stated vertex from either $\Th(1,2,3)$ or $\Th(2,2,2)$ leaves $C_4$. For every other vertex in these two graphs, the deletion either has no perfect matching or has maximal matchings of two different sizes. In $\Th(1,3,4)$ and in the graph obtained from $\Th(1,2,2)$ by attaching one leaf to an internal vertex of a length-$2$ path, checking one representative from each vertex orbit gives the same conclusion. After pendant $2$-paths are attached at $w$, deleting $w$ leaves $C_4$ together with isolated $K_2$ components, whereas deleting a support vertex isolates its leaf and deleting any other vertex leaves a nonperfect maximal matching involving an attachment edge.
\end{proof}

\begin{theorem}\label{thm:A-odd}
Let $G\in\Aclass_n$, where $n\ge5$ is odd. Then
\[
 \avm(G)\le\frac{n-1}{2}.
\]
Equality holds if and only if
\[
 G\cong\AU_{(n-5)/2},\qquad G\cong\AV_{(n-5)/2},
\]
or $n=5$ and $G$ is obtained from $\Th(1,2,2)$ by attaching one leaf to an internal vertex of a length-$2$ path, or $n=7$ and $G\cong\Th(1,3,4)$.
\end{theorem}

\begin{proof}
Only equality needs proof. By \cref{lem:odd-support,lem:farthest}, repeatedly delete a degree-$2$ support vertex outside the core together with its leaf. By \cref{lem:P2,cor:odd-extension}, the smaller graph remains an odd-order extremal graph, and at each reverse step the attachment vertex must have a randomly matchable deletion. The process terminates at a reduced graph, which is one of the four graphs in \cref{lem:A-reduced-odd}. The admissible vertices in \cref{lem:A-odd-vertices} give exactly $\AU_r,\AV_r$, while the two exceptional graphs of orders $5$ and $7$ cannot be extended. Conversely, repeated use of \cref{cor:odd-extension} proves that every listed graph is equimatchable with maximum size $(n-1)/2$.
\end{proof}

\subsection{Even order}

\begin{lemma}\label{lem:A-three-pm}
A graph in $\Aclass_{2k}$ has at most three perfect matchings. If it has three, the three pairwise symmetric differences are exactly the three cycles of the theta core, and the perfect matchings coincide outside the core.
\end{lemma}

\begin{proof}
Fix a perfect matching $M$. The symmetric difference with another perfect matching is a disjoint union of alternating cycles. Since all cycles of a theta graph pairwise intersect, only one core cycle can occur. At a branch vertex, $M$ uses at most one of the three first path edges, so at most two of the three core cycles can be alternating relative to $M$. Thus there are at most two further perfect matchings. If both exist, their symmetric difference is the third core cycle, and no edge outside the core can differ.
\end{proof}

\begin{proposition}\label{prop:A-two-pm}
If $G\in\Aclass_{2k}$ has exactly two perfect matchings, then
\[
 m(G)\ge2k-1.
\]
\end{proposition}

\begin{proof}
Let $M,M'$ be the perfect matchings. Their symmetric difference is an even core cycle. Each has $k+1$ nonmatching edges. If the cycle is a $4$-cycle, each side has two nonmaximal switches, and \cref{lem:switch-distance} shows that the two switching families intersect in at most one matching; hence they yield at least $2(k+1)-4-1=2k-3$ nonperfect maximal matchings. If the cycle is a $6$-cycle, all switches are maximal and \cref{lem:switch-distance} gives at most three overlaps, yielding $2k-1$ nonperfect maximal matchings. For a longer cycle, the switching families are disjoint by \cref{lem:switch-distance}. Adding the two perfect matchings gives the assertion.
\end{proof}

\begin{lemma}\label{lem:A-three-paths}
If a pure theta graph $T=\Th(a,b,c)$ has three perfect matchings, then $a,b,c$ are odd.
\end{lemma}

\begin{proof}
The three pairwise symmetric differences are the three cycles, so $a+b,a+c,b+c$ are even. Thus $a,b,c$ have the same parity. Since $|V(T)|=a+b+c-1$ is even, their common parity is odd.
\end{proof}

\begin{lemma}\label{lem:A-pure-three}
Let $T=\Th(a,b,c)$ have $2k$ vertices and three perfect matchings. Then
\[
 m(T)\ge3k-4,
\]
with equality if and only if $T\cong\Th(1,3,3)$.
\end{lemma}

\begin{proof}
Let the perfect matchings be $M_1,M_2,M_3$. Counting each switch together with its originating perfect matching gives $3(k+1)$ ordered switches. A pair whose symmetric difference is a $4$-cycle accounts for four nonmaximal switches. A pair differing on a $6$-cycle can create at most three overlaps; pairs differing on longer cycles create no overlap. Let $t_4,t_6$ be the number of pairs with differences of lengths $4,6$. Then the number of distinct nonperfect maximal matchings obtained by switching is at least
\[
 3(k+1)-4t_4-3t_6.
\]
By \cref{lem:A-three-paths}, the path lengths are odd. If one pair differs on a $4$-cycle, two path lengths are $1$ and $3$, and simplicity prevents a second length-$1$ path; unless $(a,b,c)=(1,3,3)$, at most one further pair differs on a $6$-cycle. If no pair differs on a $4$-cycle, then $t_6\le3$. Both cases give strictly more than $3k-7$ nonperfect maximal matchings, except for $\Th(1,3,3)$.

For $\Th(1,3,3)$, writing the long paths as $uabv$ and $ucdv$, the three perfect matchings are
\[
 \{uv,ab,cd\},\quad \{ua,bv,cd\},\quad \{uc,dv,ab\},
\]
and the only other maximal matchings are $\{ua,dv\}$ and $\{uc,bv\}$. Hence $m=5=3k-4$.
\end{proof}

\begin{proposition}\label{prop:A-three-all}
If $G\in\Aclass_{2k}$ has three perfect matchings, then
\[
 m(G)\ge3k-4,
\]
with equality if and only if $G\cong\AW_{k-3}$.
\end{proposition}

\begin{proof}
Induct on $k$. The pure-core case is \cref{lem:A-pure-three}. If $G$ is not pure, a support vertex cannot lie only on the core: a pendant edge would be fixed in all perfect matchings, but every core vertex belongs to at least two of the three alternating cycles. Hence \cref{lem:farthest} gives a degree-$2$ support vertex $s$ outside the core, with leaf $t$ and other neighbor $w$. Put $H=G-\{s,t\}$. Then $H$ still has three perfect matchings, and
\[
 m(G)=m(H)+m(H-w).
\]
For every perfect matching of $H$, deleting its edge covering $w$ produces a distinct maximal matching of $H-w$, so $m(H-w)\ge3$. The induction hypothesis gives
\[
 m(G)\ge[3(k-1)-4]+3=3k-4.
\]
Equality forces $H$ to attain equality and $m(H-w)=3$. In $\AW_r$, deleting the branch vertex that supports all pendant $2$-paths leaves $P_5\cup rK_2$, which has exactly three maximal matchings. For $r\ge1$, every other vertex gives more after deletion. For $r=0$, the two branch vertices are symmetric, and attaching at either one produces a graph isomorphic to $\AW_1$. Thus the reverse construction yields precisely $\AW_{k-3}$.
\end{proof}

\begin{lemma}\label{lem:A-W-dist}
Let $n=6+2r=2k$. The graph $\AW_r$ has exactly three perfect matchings and $3r+2=3k-7$ maximal matchings of size $k-1$. Therefore
\[
 m(\AW_r)=3k-4=\frac{3n-8}{2}
\]
and
\[
 \avm(\AW_r)=k-1+\frac{3}{3k-4}
 =\frac n2-1+\frac{6}{3n-8}.
\]
\end{lemma}

\begin{proof}
For $r=0$, the three perfect matchings and the two maximal matchings of size $2$ were listed in the proof of \cref{lem:A-pure-three}. Adding a pendant $2$-path at the branch vertex $u$ leaves, after deleting $u$, the graph $P_5\cup rK_2$, which has exactly three maximal matchings. Thus \cref{lem:P2} shows that each added pendant $2$-path preserves the three perfect matchings and adds three maximal matchings whose size is one less than that of a perfect matching.
\end{proof}

\begin{theorem}\label{thm:A-even}
Let $G\in\Aclass_n$, where $n=2k\ge6$. Then
\[
 \avm(G)\le k-1+\frac{3}{3k-4}
 =\frac n2-1+\frac{6}{3n-8}.
\]
Equality holds if and only if $G\cong\AW_{k-3}$; see Figure~\ref{fig:A-extremal}(\ref{fig:A-W}).
\end{theorem}

\begin{proof}
By \cref{lem:A-three-pm}, $p=\pc(G)\le3$. If $p=0$, then $\avm(G)\le k-1$. If $p=1$, the $k+1$ nonmatching edges of the unique perfect matching all produce distinct maximal switches, so $m(G)\ge k+2$ and
\[
 \avm(G)\le k-1+\frac1{k+2}<k-1+\frac3{3k-4}.
\]
If $p=2$, \cref{prop:A-two-pm} gives
\[
 \avm(G)\le k-1+\frac2{2k-1}<k-1+\frac3{3k-4}.
\]
If $p=3$, \cref{prop:A-three-all} gives $m(G)\ge3k-4$, hence
\[
 \avm(G)\le k-1+\frac3{m(G)}\le k-1+\frac3{3k-4}.
\]
Equality forces $G\cong\AW_{k-3}$, and \cref{lem:A-W-dist} verifies equality.
\end{proof}

\section{Upper bounds in the class \texorpdfstring{$\mathcal{B}(p,q)$}{B(p,q)}}\label{sec:B}

Let $G\in\Bclass_n$, and let $u$ be the common vertex of the two core cycles. For $r\ge0$, let $\BO_r$ be obtained from $F_{3,3}$ by attaching $r$ pendant $2$-paths at $u$, and let $\BE_r$ be obtained from $F_{3,4}$ by attaching $r$ pendant $2$-paths at $u$. Their orders are $5+2r$ and $6+2r$, respectively; see Figure~\ref{fig:B-extremal}(\ref{fig:B-O}) and Figure~\ref{fig:B-extremal}(\ref{fig:B-E}). The two exceptional graphs that occur in the equality statements are $F_{4,4}$ and the graph obtained from $F_{3,4}$ by attaching one leaf to each of the two noncommon vertices of its triangle; see Figure~\ref{fig:B-extremal}(\ref{fig:B-F44}) and Figure~\ref{fig:B-extremal}(\ref{fig:B-exception-eight}).

\begin{figure}[htbp]
\centering
\subcaptionbox{$\BO_r\label{fig:B-O}$}[0.31\textwidth]{%
\begin{tikzpicture}[scale=0.82]
\node[vertex] (u) at (0,0) {};
\node[vertex] (a) at (-1.1,0.8) {};
\node[vertex] (b) at (-1.1,-0.8) {};
\node[vertex] (c) at (1.1,0.8) {};
\node[vertex] (d) at (1.1,-0.8) {};
\draw[edge] (u)--(a)--(b)--(u);
\draw[edge] (u)--(c)--(d)--(u);
\node[vertex] (s1) at (-0.55,-1.55) {};
\node[vertex] (t1) at (-0.95,-2.15) {};
\node[vertex] (s2) at (0.55,-1.55) {};
\node[vertex] (t2) at (0.95,-2.15) {};
\draw[edge] (u)--(s1)--(t1);
\draw[edge] (u)--(s2)--(t2);
\node at (0,-1.65) {$\cdots$};
\end{tikzpicture}}
\hfill
\subcaptionbox{$F_{4,4}\label{fig:B-F44}$}[0.31\textwidth]{%
\begin{tikzpicture}[scale=0.82]
\node[vertex] (u) at (0,0) {};
\node[vertex] (a) at (-0.9,0.8) {};
\node[vertex] (b) at (-1.8,0) {};
\node[vertex] (c) at (-0.9,-0.8) {};
\node[vertex] (d) at (0.9,0.8) {};
\node[vertex] (e) at (1.8,0) {};
\node[vertex] (f) at (0.9,-0.8) {};
\draw[edge] (u)--(a)--(b)--(c)--(u);
\draw[edge] (u)--(d)--(e)--(f)--(u);
\end{tikzpicture}}
\hfill
\subcaptionbox{$\BE_r\label{fig:B-E}$}[0.31\textwidth]{%
\begin{tikzpicture}[scale=0.82]
\node[vertex] (u) at (0,0) {};
\node[vertex] (a) at (-1.1,0.75) {};
\node[vertex] (b) at (-1.1,-0.75) {};
\node[vertex] (c) at (0.8,0.85) {};
\node[vertex] (d) at (1.8,0) {};
\node[vertex] (e) at (0.8,-0.85) {};
\draw[edge] (u)--(a)--(b)--(u);
\draw[edge] (u)--(c)--(d)--(e)--(u);
\node[vertex] (s1) at (-0.55,-1.55) {};
\node[vertex] (t1) at (-0.95,-2.15) {};
\node[vertex] (s2) at (0.55,-1.55) {};
\node[vertex] (t2) at (0.95,-2.15) {};
\draw[edge] (u)--(s1)--(t1);
\draw[edge] (u)--(s2)--(t2);
\node at (0,-1.65) {$\cdots$};
\end{tikzpicture}}

\par\medskip
\subcaptionbox{The graph obtained from $F_{3,4}$ by attaching one leaf to each of the two noncommon vertices of its triangle\label{fig:B-exception-eight}}[0.62\textwidth]{%
\begin{tikzpicture}[scale=0.84]
\node[vertex] (u) at (0,0) {};
\node[vertex] (a) at (-1.1,0.75) {};
\node[vertex] (b) at (-1.1,-0.75) {};
\node[vertex] (la) at (-1.95,1.25) {};
\node[vertex] (lb) at (-1.95,-1.25) {};
\node[vertex] (c) at (0.85,0.85) {};
\node[vertex] (d) at (1.85,0) {};
\node[vertex] (e) at (0.85,-0.85) {};
\draw[edge] (u)--(a)--(b)--(u);
\draw[edge] (a)--(la);
\draw[edge] (b)--(lb);
\draw[edge] (u)--(c)--(d)--(e)--(u);
\end{tikzpicture}}
\caption{The graphs $\BO_r$, $F_{4,4}$, $\BE_r$, and the exceptional graph of order $8$. In each parameterized graph, the fan with an ellipsis represents $r$ pendant $2$-paths; it is absent when $r=0$.}
\label{fig:B-extremal}
\end{figure}
\FloatBarrier

\subsection{Odd order}

Write
\[
 G-u=T_1\cup\cdots\cup T_s,
\qquad S_i=N_G(u)\cap V(T_i).
\]
Exactly two components correspond to the two core cycles and have $|S_i|=2$; every tree attached directly at $u$ gives a component with $|S_i|=1$.

\begin{lemma}\label{lem:B-branch}
Suppose $n=2k+1$ and every maximal matching of $G$ has size $k$. Then every $T_i$ and every $T_i-x$, $x\in S_i$, is equimatchable, and
\[
 1+\nu(T_j-x)+\sum_{i\ne j}\nu(T_i)=k
\]
for every $j$ and every $x\in S_j$.
\end{lemma}

\begin{proof}
Fix $j$ and $x\in S_j$. By \cref{lem:cut-splice}, joining $ux$ to arbitrary maximal matchings of $T_j-x$ and of the remaining components produces a maximal matching of $G$. Varying one component at a time shows that $T_j-x$ and every $T_i$ with $i\ne j$ are equimatchable. Since there are two core components, the index $j$ can be chosen different from any prescribed $i$, and hence every $T_i$ is equimatchable. Their maximal matchings therefore have sizes equal to their matching numbers, and \cref{lem:cut-splice} gives the displayed identity.
\end{proof}

For a graph $H$, define $\delta(H)=|V(H)|-2\nu(H)$.

\begin{lemma}\label{lem:B-deficiency}
Under the assumptions of \cref{lem:B-branch}, put $d_i=\delta(T_i)$ and $D=\sum_i d_i$. Then either
\begin{enumerate}[label=\textup{(\alph*)}]
\item $D=0$, so every $T_i$ has a perfect matching; or
\item $D=2$, $s=2$, $d_1=d_2=1$, and $T_i-x$ has a perfect matching for every $x\in S_i$.
\end{enumerate}
\end{lemma}

\begin{proof}
Since $G-u$ has $2k$ vertices, \cref{lem:B-branch} gives, for every $j$ and $x\in S_j$,
\[
 \delta(T_j-x)+D-d_j=1.
\]
Hence $D-d_j\le1$ for all $j$. If $D\ge3$, then $d_j\ge D-1$ for all $j$, so
\[
 D=\sum_jd_j\ge s(D-1)\ge2(D-1)>D,
\]
a contradiction. Thus $D\le2$. If $D=1$, exactly one component has deficiency $1$ and every other component has deficiency $0$. Choose $j$ among the latter components. The displayed identity gives $\delta(T_j-x)=0$, which is impossible because $T_j$ has even order and $T_j-x$ has odd order. If $D=2$, then every $d_j\ge1$, whence $s=2$ and $d_1=d_2=1$; the identity then yields $\delta(T_i-x)=0$.
\end{proof}

\begin{lemma}\label{lem:B-perfect-tree}
A connected equimatchable tree with a perfect matching is $K_2$.
\end{lemma}

\begin{proof}
Suppose $T$ has at least four vertices, and let $v_1v_2\cdots v_t$ be a longest path. Every perfect matching contains $v_1v_2$, so $v_2$ has no leaf neighbor other than $v_1$. By the choice of the path, every neighbor of $v_2$ outside the path would be a leaf; hence $d_T(v_2)=2$. Extend $v_2v_3$ to a maximal matching. This matching leaves $v_1$ uncovered and is therefore not perfect, contradicting equimatchability.
\end{proof}

\begin{theorem}\label{thm:B-odd}
Let $G\in\Bclass_n$, where $n\ge5$ is odd. Then
\[
 \avm(G)\le\frac{n-1}{2}.
\]
Equality holds if and only if
\[
 G\cong\BO_{(n-5)/2},
\]
or $n=7$ and $G\cong F_{4,4}$.
\end{theorem}

\begin{proof}
Assume equality and write $n=2k+1$. Apply \cref{lem:B-deficiency}.

If $D=0$, every $T_i$ is an equimatchable tree with a perfect matching, so $T_i=K_2$ by \cref{lem:B-perfect-tree}. The two core components, each having two neighbors of $u$, form triangles. Every other $K_2$ component has exactly one endpoint adjacent to $u$; otherwise a third cycle would be created. Thus $G=\BO_r$ with $r=(n-5)/2$.

If $D=2$, then $s=2$. For each $i$ and each $x\in S_i$, the forest $T_i-x$ is equimatchable and has a perfect matching, so every component is $K_2$ by \cref{lem:B-perfect-tree}. Writing $S_i=\{a,b\}$, \cref{lem:tree-two-deletions} gives $T_i\cong P_3$ with $a$ and $b$ as its endvertices. Hence both core cycles are $C_4$, and $G\cong F_{4,4}$.

Conversely, deleting $u$ from $\BO_r$ gives $r+2$ copies of $K_2$, and every maximal matching has size $r+2$. A direct enumeration shows that every maximal matching of $F_{4,4}$ has size $3$.
\end{proof}

\subsection{Even order}

\begin{lemma}\label{lem:B-two-pm}
A graph in $\Bclass_{2k}$ has at most two perfect matchings.
\end{lemma}

\begin{proof}
Fix a perfect matching $M$. The symmetric difference with another perfect matching must be one core cycle. If two distinct alternatives existed, the two core cycles would both be alternating relative to $M$, forcing $M$ to use two edges incident with their common vertex $u$, impossible.
\end{proof}

\begin{lemma}\label{lem:B-one-pm-count}
If $G\in\Bclass_{2k}$ has one perfect matching, then $m(G)\ge k+2$ and
\[
 \avm(G)<k-1+\frac1k.
\]
\end{lemma}

\begin{proof}
The unique perfect matching has $k+1$ nonmatching edges. Every switch is maximal, since a nonmaximal switch would complete to a second perfect matching. The switches are distinct, so $m(G)\ge k+2$, and
\[
 \avm(G)\le k-1+\frac1{m(G)}\le k-1+\frac1{k+2}<k-1+\frac1k.
\]
\end{proof}

\begin{lemma}\label{lem:B-two-count}
If $G\in\Bclass_{2k}$ has two perfect matchings, then it has at least $2k-2$ nonperfect maximal matchings and hence $m(G)\ge2k$.
\end{lemma}

\begin{proof}
Let $M\triangle M'$ be an alternating core cycle of length $2r$. For $r\ge4$, the two switching families are disjoint by \cref{lem:switch-distance} and give $2k+2$ nonperfect maximal matchings. For $r=3$, \cref{lem:switch-distance} gives at least $2(k+1)-3=2k-1$. For $r=2$, each family loses the two switches that complete to the other perfect matching, leaving $k-1$ maximal switches in each family. The alternating core $4$-cycle is chordless, so the two switching families are disjoint by \cref{lem:switch-distance}. Thus there are $2k-2$ nonperfect maximal matchings.
\end{proof}

\begin{proposition}\label{prop:B-even-bound}
For $G\in\Bclass_{2k}$,
\[
 \avm(G)\le k-1+\frac1k.
\]
If equality holds, then $G$ has two perfect matchings differing on a core $4$-cycle, $m(G)=2k$, and every nonperfect maximal matching has size $k-1$.
\end{proposition}

\begin{proof}
Use \cref{lem:B-two-pm}. The cases of zero and one perfect matching are strict. With two perfect matchings and $N=m(G)$,
\[
 \avm(G)\le k-1+\frac2N\le k-1+\frac1k
\]
by \cref{lem:B-two-count}. Equality forces equality in every step.
\end{proof}

Assume equality, let $M,M'$ be the two perfect matchings, let $C=M\triangle M'$ be the alternating $4$-cycle, and define
\[
 F=M\cap M',
 \qquad
 X=E(G)\setminus(M\cup M').
\]
Then $|F|=k-2$ and $|X|=k-1$. For each $e\in X$, switching from $M$ and from $M'$ yields two distinct maximal matchings. The following criterion determines when these $2|X|=2k-2$ matchings are all the nonperfect maximal matchings.

\begin{lemma}\label{lem:B-X}
Under the equality hypotheses of \cref{prop:B-even-bound}, $m(G)=2k$ if and only if every two edges of $X$ are incident.
\end{lemma}

\begin{proof}
If $e,f\in X$ are disjoint, extend $\{e,f\}$ to a maximal matching. It contains at least two edges of $X$, so it is neither a perfect matching nor a single switch, giving more than $2k$ maximal matchings.

Conversely, if $X$ is pairwise intersecting, a matching contains at most one edge of $X$. A maximal matching containing none must contain all edges of $F$ and one of the two perfect matchings of $C$, hence is $M$ or $M'$. A maximal matching containing $e\in X$ must contain every $F$-edge disjoint from $e$ and one of the two maximal choices on $C$; it is exactly the switch from $M$ or from $M'$. Thus there are $2+2|X|=2k$ maximal matchings.
\end{proof}

Let the second core cycle have neighbors $a,b$ of $u$. Since $u$ is matched inside $C$ by both perfect matchings, $ua,ub\in X$.

\begin{lemma}\label{lem:B-X-shape}
If $X$ is pairwise intersecting, then either all edges of $X$ are incident with $u$, or
\[
 X=\{ua,ub,ab\}.
\]
\end{lemma}

\begin{proof}
An edge that meets both $ua$ and $ub$ either contains $u$ or is $ab$. If $ab\in X$, no further edge $uc$ with $c\notin\{a,b\}$ can belong to $X$.
\end{proof}

\begin{lemma}\label{lem:B-equality-graphs}
The first case of \cref{lem:B-X-shape} gives $G\cong\BE_{k-3}$. The second gives $k=4$ and the graph obtained from $F_{3,4}$ by attaching one leaf to each of the two noncommon vertices of its triangle.
\end{lemma}

\begin{proof}
If all edges of $X$ meet $u$, the second core cycle must be a triangle: every edge of the path between $a$ and $b$ lies outside $X$ and hence in the common matching $F$, so the path can have only one edge. Every remaining component of $G-u$ is a $K_2$: after an edge from $u$ in $X$, the next edge is forced to lie in the common matching $F$, and any further edge would belong to $X$ without being incident with $u$. Thus every such component forms a pendant $2$-path at $u$, and $G\cong\BE_{k-3}$.

If $X=\{ua,ub,ab\}$, then $|X|=k-1=3$, so $k=4$. The three edges form the second core triangle. The common matching $F$ has two edges and must cover $a$ and $b$ by distinct edges. The other endpoints of these edges are leaves, since any further incident edge would belong to $X$. Hence $G$ is the graph obtained from $F_{3,4}$ by attaching one leaf to each of the two noncommon vertices of its triangle.
\end{proof}

\begin{theorem}\label{thm:B-even}
Let $G\in\Bclass_n$, where $n=2k\ge6$. Then
\[
 \avm(G)\le k-1+\frac1k
 =\frac n2-1+\frac2n.
\]
Equality holds if and only if $G\cong\BE_{k-3}$, or $n=8$ and $G$ is obtained from $F_{3,4}$ by attaching one leaf to each of the two noncommon vertices of its triangle. These graphs are shown in Figure~\ref{fig:B-extremal}(\ref{fig:B-E}) and Figure~\ref{fig:B-extremal}(\ref{fig:B-exception-eight}).
\end{theorem}

\begin{proof}
The bound is \cref{prop:B-even-bound}, and the equality classification is \cref{lem:B-X,lem:B-X-shape,lem:B-equality-graphs}. In either extremal graph there are two perfect matchings and $2k-2$ maximal matchings of size $k-1$, so the displayed average follows.
\end{proof}

\section{Upper bounds in the class \texorpdfstring{$\mathcal{C}(p,q)$}{C(p,q)}}\label{sec:C}

Let $G\in\Cclass_n$. Its core is $\Dumb_\ell(p,q)$, with connector
\[
 a_0=x_0,x_1,\ldots,x_{\ell-1},x_\ell=b_0.
\]
For $r\ge0$, define the following four families. The graph $\CO_r$ is obtained from $\Dumb_1(3,4)$ by attaching $r$ pendant $2$-paths at the connector endpoint on the triangle; $\CP_r$ is obtained from $\Dumb_2(4,4)$ by attaching $r$ pendant $2$-paths at the internal connector vertex; $\CQ_r$ is obtained from $\Dumb_1(4,4)$ by attaching $r$ pendant $2$-paths at one connector endpoint; and $\CR_r$ is obtained from $\Dumb_2(4,4)$ by attaching one leaf and $r$ pendant $2$-paths at the internal connector vertex. Their orders are $7+2r$, $9+2r$, $8+2r$, and $10+2r$, respectively; see Figure~\ref{fig:C-extremal}(\ref{fig:C-O})--(\ref{fig:C-R}). For later reference, Figure~\ref{fig:C-extremal}(\ref{fig:C-D133}) also shows $\Dumb_1(3,3)$, the unique graph in $\Cclass_6$.

\begin{figure}[htbp]
\centering
\subcaptionbox{$\CO_r\label{fig:C-O}$}[0.31\textwidth]{%
\begin{tikzpicture}[scale=0.8]
\node[vertex] (a0) at (-0.1,0) {};
\node[vertex] (a1) at (-1.0,0.75) {};
\node[vertex] (a2) at (-1.0,-0.75) {};
\node[vertex] (b0) at (1.25,0) {};
\node[vertex] (b1) at (2.15,0.8) {};
\node[vertex] (b2) at (3.0,0) {};
\node[vertex] (b3) at (2.15,-0.8) {};
\draw[edge] (a0)--(a1)--(a2)--(a0);
\draw[edge] (a0)--(b0);
\draw[edge] (b0)--(b1)--(b2)--(b3)--(b0);
\node[vertex] (s1) at (-0.65,1.4) {};
\node[vertex] (t1) at (-1.15,1.95) {};
\node[vertex] (s2) at (0.4,1.4) {};
\node[vertex] (t2) at (0.9,1.95) {};
\draw[edge] (a0)--(s1)--(t1);
\draw[edge] (a0)--(s2)--(t2);
\node at (-0.1,1.52) {$\cdots$};
\end{tikzpicture}}
\hfill
\subcaptionbox{$\CP_r\label{fig:C-P}$}[0.31\textwidth]{%
\begin{tikzpicture}[scale=0.8]
\node[vertex] (a0) at (-0.1,0) {};
\node[vertex] (a1) at (-1.0,0.8) {};
\node[vertex] (a2) at (-1.85,0) {};
\node[vertex] (a3) at (-1.0,-0.8) {};
\node[vertex] (x) at (1.1,0) {};
\node[vertex] (b0) at (2.3,0) {};
\node[vertex] (b1) at (3.2,0.8) {};
\node[vertex] (b2) at (4.05,0) {};
\node[vertex] (b3) at (3.2,-0.8) {};
\draw[edge] (a0)--(a1)--(a2)--(a3)--(a0);
\draw[edge] (a0)--(x)--(b0);
\draw[edge] (b0)--(b1)--(b2)--(b3)--(b0);
\node[vertex] (s1) at (0.55,1.0) {};
\node[vertex] (t1) at (0.05,1.6) {};
\node[vertex] (s2) at (1.65,1.0) {};
\node[vertex] (t2) at (2.15,1.6) {};
\draw[edge] (x)--(s1)--(t1);
\draw[edge] (x)--(s2)--(t2);
\node at (1.1,1.1) {$\cdots$};
\end{tikzpicture}}
\hfill
\subcaptionbox{$\CQ_r\label{fig:C-Q}$}[0.31\textwidth]{%
\begin{tikzpicture}[scale=0.8]
\node[vertex] (a0) at (0,0) {};
\node[vertex] (a1) at (-0.9,0.8) {};
\node[vertex] (a2) at (-1.75,0) {};
\node[vertex] (a3) at (-0.9,-0.8) {};
\node[vertex] (b0) at (1.45,0) {};
\node[vertex] (b1) at (2.35,0.8) {};
\node[vertex] (b2) at (3.2,0) {};
\node[vertex] (b3) at (2.35,-0.8) {};
\draw[edge] (a0)--(a1)--(a2)--(a3)--(a0);
\draw[edge] (a0)--(b0);
\draw[edge] (b0)--(b1)--(b2)--(b3)--(b0);
\node[vertex] (s1) at (-0.45,1.45) {};
\node[vertex] (t1) at (-0.85,2.05) {};
\node[vertex] (s2) at (0.45,1.45) {};
\node[vertex] (t2) at (0.85,2.05) {};
\draw[edge] (a0)--(s1)--(t1);
\draw[edge] (a0)--(s2)--(t2);
\node at (0,1.55) {$\cdots$};
\end{tikzpicture}}

\par\medskip
\subcaptionbox{$\CR_r\label{fig:C-R}$}[0.47\textwidth]{%
\begin{tikzpicture}[scale=0.8]
\node[vertex] (a0) at (-0.2,0) {};
\node[vertex] (a1) at (-1.1,0.8) {};
\node[vertex] (a2) at (-1.95,0) {};
\node[vertex] (a3) at (-1.1,-0.8) {};
\node[vertex] (x) at (1.0,0) {};
\node[vertex] (b0) at (2.2,0) {};
\node[vertex] (b1) at (3.1,0.8) {};
\node[vertex] (b2) at (3.95,0) {};
\node[vertex] (b3) at (3.1,-0.8) {};
\node[vertex] (leaf) at (1.0,-1.15) {};
\draw[edge] (a0)--(a1)--(a2)--(a3)--(a0);
\draw[edge] (a0)--(x)--(b0);
\draw[edge] (b0)--(b1)--(b2)--(b3)--(b0);
\draw[edge] (x)--(leaf);
\node[vertex] (s1) at (0.45,1.0) {};
\node[vertex] (t1) at (-0.05,1.6) {};
\node[vertex] (s2) at (1.55,1.0) {};
\node[vertex] (t2) at (2.05,1.6) {};
\draw[edge] (x)--(s1)--(t1);
\draw[edge] (x)--(s2)--(t2);
\node at (1.0,1.1) {$\cdots$};
\end{tikzpicture}}
\hfill
\subcaptionbox{$\Dumb_1(3,3)\label{fig:C-D133}$}[0.31\textwidth]{%
\begin{tikzpicture}[scale=0.84]
\node[vertex] (a0) at (-0.1,0) {};
\node[vertex] (a1) at (-1.0,0.8) {};
\node[vertex] (a2) at (-1.0,-0.8) {};
\node[vertex] (b0) at (1.2,0) {};
\node[vertex] (b1) at (2.1,0.8) {};
\node[vertex] (b2) at (2.1,-0.8) {};
\draw[edge] (a0)--(a1)--(a2)--(a0);
\draw[edge] (a0)--(b0);
\draw[edge] (b0)--(b1)--(b2)--(b0);
\end{tikzpicture}}
\caption{The graphs $\CO_r$, $\CP_r$, $\CQ_r$, $\CR_r$, and $\Dumb_1(3,3)$. In each parameterized graph, the fan with an ellipsis represents $r$ pendant $2$-paths; it is absent when $r=0$.}
\label{fig:C-extremal}
\end{figure}
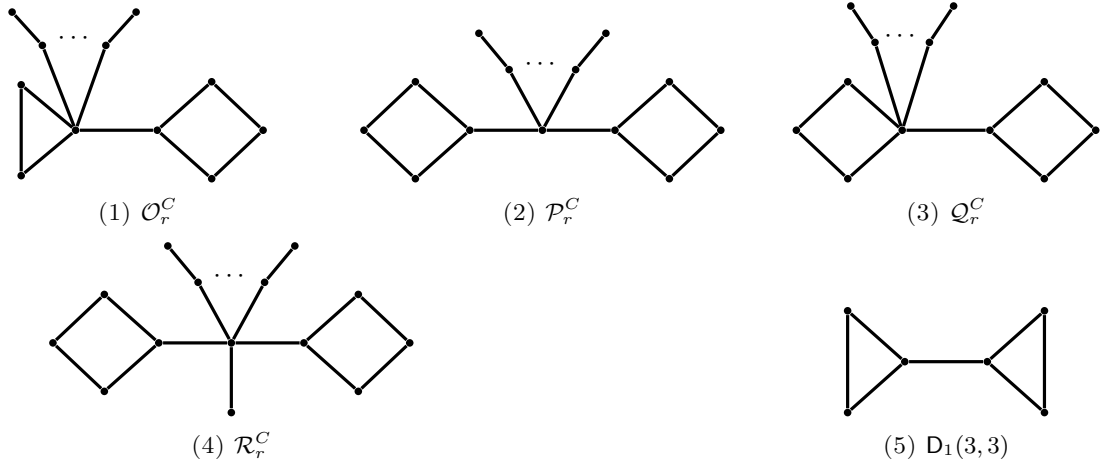
\FloatBarrier

\subsection{Odd order}

We first record two elementary equimatchability classifications.

\begin{lemma}\label{lem:path-equi}
The path $P_t$ is equimatchable if and only if $t\in\{1,2,3,5\}$.
\end{lemma}

\begin{proof}
The maximum matching size is $\lfloor t/2\rfloor$, while the minimum size of a maximal matching is $\lceil(t-1)/3\rceil$. Equality holds precisely for $t=1,2,3,5$.
\end{proof}

For $t\ge1$, let $L(p,t)$ be obtained by identifying one endpoint of a path of length $t$ with a vertex of $C_p$; set $L(p,0)=C_p$.

\begin{lemma}\label{lem:lollipop}
The graph $L(p,t)$ is equimatchable if and only if
\[
 (p,t)\in\{(3,0),(3,2),(4,0),(4,1),(4,3),(5,0),(7,0)\}.
\]
\end{lemma}

\begin{proof}
For $t=0$, the cycle $C_p$ is equimatchable exactly when
\[
 \left\lceil\frac p3\right\rceil=\left\lfloor\frac p2\right\rfloor,
\]
namely for $p=3,4,5,7$. For $t\ge1$, delete the cut vertex shared by the cycle and the tail. The components $P_{p-1}$ and $P_t$ must both be equimatchable, so by \cref{lem:path-equi}, $p\in\{3,4,6\}$ and $t\in\{1,2,3,5\}$. Direct checking of these twelve cases leaves $(3,2),(4,1),(4,3)$.
\end{proof}

\begin{lemma}\label{lem:C-pure-odd}
Let $D=\Dumb_\ell(p,q)$ have odd order. If every maximal matching of $D$ is maximum, then
\[
 D\cong\Dumb_1(3,4)
 \quad\text{or}\quad
 D\cong\Dumb_2(4,4).
\]
\end{lemma}

\begin{proof}
Delete $a_0$. The components are $P_{p-1}$ and $L(q,\ell-1)$. By \cref{lem:cut-splice}, both are equimatchable; applying \cref{lem:path-equi,lem:lollipop} and then exchanging the two cycles gives $p,q\in\{3,4\}$. If $p=q=3$, the allowed connector lengths are $1$ and $3$, and both resulting graphs have even order. If $\{p,q\}=\{3,4\}$, only $\ell=1$ survives. If $p=q=4$, the odd-order condition and the lollipop list give $\ell=2$ or $4$. The graph $\Dumb_4(4,4)$ has the maximal matching
\[
 \{a_0x_1,x_3b_0,a_2a_3,b_2b_3\}
\]
of size $4$, whereas its matching number is $5$. Thus only the two stated cores remain.
\end{proof}

\begin{lemma}\label{lem:C-no-leaf}
Let $G\in\Cclass_{2k+1}$ be reduced. If every maximal matching of $G$ is maximum, then $G$ has no leaves.
\end{lemma}

\begin{proof}
By \cref{lem:odd-support}, every support vertex has exactly one leaf neighbor. If a leaf had a support vertex outside the core, then \cref{lem:farthest} would give a degree-$2$ support vertex outside the core, contradicting reducedness. Hence the support vertex $x$ of any leaf $y$ lies on the core. Put $H=G-\{x,y\}$. For every nonleaf core neighbor $z$ of $x$ and every maximal matching $N$ of $H-z$, the matching $N\cup\{xz\}$ is maximal in $G$. It has size $k$, and therefore $N$ is a perfect matching of $H-z$. Thus $H-z$ is randomly matchable, so every component of $H-z$ is $K_2$ or $C_4$ by \cref{lem:sparse-random}.

If $x$ lies on a cycle but is not a connector endpoint, choose a cycle neighbor $z$ so that the other route from $x$ to the connector endpoint remains. The component of $H-z$ containing the other core cycle contains that cycle together with at least one additional edge. It is therefore neither $K_2$ nor $C_4$, a contradiction.

Suppose next that $x=a_0$ is a connector endpoint. Choose the connector neighbor of $a_0$ for $z$. Then $C_p-a_0=P_{p-1}$ is a component of $H-z$, so it must be $K_2$ and $p=3$. If instead $z$ is one of the two triangle neighbors of $a_0$, the other triangle neighbor is isolated in $H-z$, contradicting random matchability. The case $x=b_0$ is symmetric.

Finally, suppose that $x=x_i$ is an internal connector vertex. Taking $z=x_{i-1}$ shows that the component on the $b_0$-side must be isomorphic to $C_4$, so $i=\ell-1$ and $q=4$. Symmetrically, taking $z=x_{i+1}$ gives $i=1$ and $p=4$. Hence $\ell=2$, $p=q=4$, and $x=x_1$. Taking $z=a_0$ now leaves the component $C_4-a_0=P_3$, which is neither $K_2$ nor $C_4$. This final contradiction proves the lemma.
\end{proof}

\begin{lemma}\label{lem:C-odd-attachment}
For every $r\ge0$, the only vertex $w$ such that $\CO_r-w$ is randomly matchable is the common attachment point on the triangle. For every $r\ge0$, the only such vertex in $\CP_r$ is the unique internal connector vertex.
\end{lemma}

\begin{proof}
Deleting the stated vertex from $\CO_r$ leaves $C_4$ and copies of $K_2$. Deleting it from $\CP_r$ leaves two copies of $C_4$ and copies of $K_2$. For every other vertex, a direct check of the vertex orbits shows that the deletion either creates an isolated vertex or leaves a unicyclic component that is not a $C_4$.
\end{proof}

\begin{theorem}\label{thm:C-odd}
Let $G\in\Cclass_n$, where $n\ge7$ is odd. Then
\[
 \avm(G)\le\frac{n-1}{2}.
\]
Equality holds if and only if $G\cong\CO_{(n-7)/2}$ or, for $n\ge9$, $G\cong\CP_{(n-9)/2}$. The corresponding graphs are shown in Figure~\ref{fig:C-extremal}(\ref{fig:C-O}) and Figure~\ref{fig:C-extremal}(\ref{fig:C-P}).
\end{theorem}

\begin{proof}
Under equality, repeatedly delete a degree-$2$ support vertex outside the core and its leaf. This is possible by \cref{lem:odd-support,lem:farthest}. By \cref{lem:P2,cor:odd-extension}, the smaller graph remains an odd-order extremal graph. The reduced graph has no leaf by \cref{lem:C-no-leaf}, hence is a pure core and is one of the two graphs in \cref{lem:C-pure-odd}. The admissible attachment vertices are determined by \cref{lem:C-odd-attachment}. The converse follows from \cref{cor:odd-extension}.
\end{proof}

\subsection{Distributions of the candidate families}

\begin{lemma}\label{lem:C-distributions}
For $r\ge0$,
\[
 m(\CO_r)=2r+8,
 \qquad
 m(\CP_r)=4r+12,
\]
and every maximal matching in these graphs is maximum. If $n=8+2r$, then $\CQ_r$ has four perfect matchings and $2n-12$ maximal matchings of size $n/2-1$. If $n=10+2r$, the same distribution holds for $\CR_r$. Thus
\[
 m(\CQ_r)=m(\CR_r)=2n-8
\]
and
\[
 \avm(\CQ_r)=\avm(\CR_r)=\frac n2-1+\frac2{n-4}.
\]
\end{lemma}

\begin{proof}
A direct enumeration gives $m(\CO_0)=8$ and $m(\CP_0)=12$. Deleting the common attachment vertex leaves one $C_4$ in the first family and two $C_4$'s in the second, so \cref{lem:P2} adds $2$ or $4$ maximal matchings for each added pendant $2$-path.

A direct enumeration shows that $\CQ_0=\Dumb_1(4,4)$ has four perfect matchings and four maximal matchings of size $3$, whereas $\CR_0$ has four perfect matchings and eight maximal matchings of size $4$. In $\CQ_r$, deleting the common attachment vertex leaves $P_3\cup C_4\cup rK_2$, which has exactly four maximal matchings. In $\CR_r$, deleting the common attachment vertex leaves two copies of $C_4$, one isolated vertex, and $rK_2$, and again there are exactly four maximal matchings. The recurrence in \cref{lem:P2} completes the induction.
\end{proof}

\subsection{Even order}

For a maximal matching $M$ in a graph of order $2k$, define
\[
 \defi(M)=k-|M|.
\]
Let
\[
 D(G)=\sum_{M\in\Mcal(G)}\defi(M),
\qquad
 R(G)=\sum_{\substack{M\in\Mcal(G)\\M\text{ not perfect}}}(\defi(M)-1).
\]
If $p=\pc(G)$, then
\[
 D(G)=m(G)-p+R(G)
\]
and
\[
 \avm(G)=k-\frac{D(G)}{m(G)}.
\]

\begin{lemma}\label{lem:deficit-C}
For $n=2k\ge8$, the inequality
\[
 \avm(G)\le k-1+\frac{2}{n-4}
\]
is equivalent to
\[
 2m(G)+(n-4)R(G)\ge(n-4)\pc(G).
\]
\end{lemma}

\begin{proof}
Substitute $D(G)=m(G)-p+R(G)$ into $\avm(G)=k-D(G)/m(G)$ and rearrange.
\end{proof}

The general bound that a connected $2$-cyclic graph has at most four perfect matchings follows from Engbers and Erey~\cite[Lemma~5.2]{EngbersErey2023}. For the class $\Cclass$, we need the following sharper description.

\begin{lemma}\label{lem:C-pm-values}
For $G\in\Cclass_{2k}$,
\[
 \pc(G)\in\{0,1,2,4\}.
\]
Moreover, $\pc(G)=4$ if and only if both core cycles are even and
\[
 G-\bigl(V(C_p)\cup V(C_q)\bigr)
\]
has a perfect matching.
\end{lemma}

\begin{proof}
For every bridge, whether it belongs to a perfect matching is determined by the parities of the two components obtained after deleting it. After these bridge choices are fixed, the forest outside the two core cycles has at most one compatible matching, while each core cycle contributes either one forced choice or two alternating choices. Hence $\pc(G)\in\{0,1,2,4\}$. Four perfect matchings occur exactly when both cycles are even and all vertices outside the cycles can be matched without using an edge incident with a cycle vertex; this is equivalent to the stated perfect-matching condition.
\end{proof}

\begin{proposition}\label{prop:C-less-four}
If $G\in\Cclass_{2k}$ has at most two perfect matchings, then
\[
 \avm(G)<k-1+\frac2{n-4}.
\]
\end{proposition}

\begin{proof}
With no perfect matching, $\avm(G)\le k-1$. With one perfect matching, the $k+1$ nonmatching edges all give distinct maximal switches, so $m(G)\ge k+2$. Hence
\[
 2m(G)\ge2k+4>2k-4=n-4=(n-4)\pc(G),
\]
and \cref{lem:deficit-C} is strict.

Now suppose that $G$ has exactly two perfect matchings $M,M'$. Their symmetric difference is a single alternating cycle: if it had two components, reversing either component separately would produce further perfect matchings. If this cycle has length $4$, the only nonmaximal switches are the two that complete to the other perfect matching. Hence each switching family has $k-1$ effective members, and \cref{lem:switch-distance} shows that the two families have at most one common member. Thus they produce at least $2k-3=n-3$ distinct nonperfect maximal matchings. If the cycle has length $6$, all $2(k+1)$ switches are maximal and at most three overlap, so there are at least $2k-1$ distinct nonperfect maximal matchings. For a longer cycle the two switching families are disjoint and give still more. In every case
\[
 m(G)\ge2+(n-3)=n-1>n-4.
\]
Therefore
\[
 2m(G)>2(n-4)=(n-4)\pc(G),
\]
and \cref{lem:deficit-C} again gives strict inequality.
\end{proof}

Assume now that $G$ has four perfect matchings, denoted $M_{00},M_{10},M_{01},M_{11}$ according to the independent choices on the two even cycles.

\begin{lemma}\label{lem:C-four-switch}
If $G\in\Cclass_{2k}$ has four perfect matchings, then
\[
 m(G)\ge4k-8=2n-8.
\]
Moreover, if at least one core cycle is not a $C_4$, then $m(G)>4k-8$.
\end{lemma}

\begin{proof}
Each perfect matching has $k+1$ nonmatching edges. A $C_4$ causes exactly two nonmaximal switches for each perfect matching; a longer even cycle causes none. Opposite perfect matchings differ on both cycles, so their switching families are disjoint. Adjacent families differing on a core $4$-cycle are disjoint because that cycle is chordless; adjacent families can otherwise overlap only when the relevant cycle has length $6$, and then \cref{lem:switch-distance} bounds the overlap by three.

If both cycles are $C_4$, each switching family has $k-3$ members and the four families are disjoint, so $m(G)\ge4+(4k-12)=4k-8$. If exactly one cycle is $C_4$, then $m(G)\ge4+(4k-10)=4k-6$. If neither cycle is a $C_4$, then $m(G)\ge4+(4k-8)=4k-4$. This proves both assertions.
\end{proof}

\begin{corollary}\label{cor:C-even-pre}
If $G\in\Cclass_{2k}$ has four perfect matchings, then
\[
 \avm(G)\le k-1+\frac2{n-4}.
\]
Equality implies
\[
 R(G)=0,
 \qquad
 m(G)=4k-8,
\]
and both core cycles are $C_4$.
\end{corollary}

\begin{proof}
By \cref{lem:C-four-switch}, $m(G)\ge4k-8$. Since $n=2k$ and $R(G)\ge0$,
\[
 2m(G)+(n-4)R(G)\ge2(4k-8)=4(n-4),
\]
so the bound follows from \cref{lem:deficit-C}. Equality requires $m(G)=4k-8$ and $R(G)=0$; the strict part of \cref{lem:C-four-switch} then forces both core cycles to be $C_4$.
\end{proof}

We now determine the equality cases. Assume that both core cycles are $C_4$. Define
\[
 X=E(G)\setminus\bigcup_{i,j\in\{0,1\}}M_{ij}.
\]

\begin{lemma}\label{lem:C-X-size}
We have $|X|=k-3$. For every perfect matching $M_{ij}$ and every $e\in X$, the switched matching $(M_{ij})_e$ is maximal. The resulting $4(k-3)$ matchings are pairwise distinct.
\end{lemma}

\begin{proof}
The union of the perfect matchings contains all eight cycle edges and a matching of size $k-4$ that is common to all four perfect matchings outside the two cycles. Hence its size is $k+4$, and $|X|=(2k+1)-(k+4)=k-3$. A nonmaximal switch using $e\in X$ would complete to a perfect matching containing $e$, impossible. The disjointness of the four switching families is the two-$C_4$ case of \cref{lem:C-four-switch}: opposite perfect matchings differ in eight edges, and adjacent ones differ on a chordless $C_4$.
\end{proof}

\begin{lemma}\label{lem:C-X-star}
If $m(G)=4k-8$, then the edges of $X$ are pairwise intersecting and have a common endpoint.
\end{lemma}

\begin{proof}
The four perfect matchings and the $4(k-3)$ switches in \cref{lem:C-X-size} already account for $4k-8$ maximal matchings. If $e,f\in X$ are disjoint, extend $\{e,f\}$ to a maximal matching. It contains at least two edges of $X$, so it is not among the counted matchings. Thus $X$ is pairwise intersecting.

All cycle edges lie in the union of the perfect matchings, so $X$ lies in the forest obtained by deleting the two cycle edge sets. If $|X|\le1$, the conclusion is immediate. Otherwise, choose two edges $uv,uw\in X$. Any further edge meeting both but not containing $u$ would be $vw$, which would create a triangle in the forest. Thus every edge of $X$ is incident with $u$.
\end{proof}

\begin{proposition}\label{prop:C-four-equality}
Suppose $G\in\Cclass_{2k}$ has four perfect matchings and $R(G)=0$. Then $m(G)\ge4k-8$, with equality if and only if
\[
 G\cong\CQ_{k-4}
 \quad\text{or}\quad
 G\cong\CR_{k-5}.
\]
\end{proposition}

\begin{proof}
Assume equality. By \cref{cor:C-even-pre}, both cycles are $C_4$. The first and last connector edges belong to $X$, because their cycle endpoints are matched inside the cycles in all four perfect matchings. If $\ell\ge3$, these two edges are disjoint, contradicting \cref{lem:C-X-star}. Thus $\ell\le2$.

If $\ell=1$, the connector edge $a_0b_0$ belongs to $X$. The common center of the star $X$ is one endpoint, say $a_0$. Every tree attached outside the core starts with an edge not used by any perfect matching and hence with an edge of $X$, so all such trees attach at $a_0$. Let $a_0z\in X$ be the first edge of one of these trees. The vertex $z$ is covered in every perfect matching by the same edge $zz'$. Neither $z$ nor $z'$ has any additional neighbor in the attached tree, since the corresponding edge would belong to $X$ without being incident with $a_0$. Thus every attached tree is a pendant $2$-path, and $G\cong\CQ_{k-4}$.

If $\ell=2$, both $a_0x_1$ and $x_1b_0$ lie in $X$, so the common center is $x_1$. The vertex $x_1$ must be covered in every perfect matching by a common edge $x_1y$; the star property of $X$ then forces $y$ to be a leaf. Every other tree attached outside the core is a pendant $2$-path at $x_1$. Hence $G\cong\CR_{k-5}$.

Conversely, \cref{lem:C-distributions} verifies $m=4k-8$ and $R=0$ for both families.
\end{proof}

\begin{theorem}\label{thm:C-even}
Let $G\in\Cclass_n$, where $n\ge8$ is even. Then
\[
 \avm(G)\le\frac n2-1+\frac2{n-4}.
\]
Equality holds if and only if $G\cong\CQ_{(n-8)/2}$ or, for $n\ge10$, $G\cong\CR_{(n-10)/2}$. These two extremal families are shown in Figure~\ref{fig:C-extremal}(\ref{fig:C-Q}) and Figure~\ref{fig:C-extremal}(\ref{fig:C-R}).
\end{theorem}

\begin{proof}
By \cref{lem:C-pm-values}, the number of perfect matchings is $0,1,2$, or $4$. The first three cases are strict by \cref{prop:C-less-four}. The four-perfect-matching case follows from \cref{cor:C-even-pre,prop:C-four-equality}.
\end{proof}

\begin{lemma}\label{lem:C-six}
The family $\Cclass_6$ contains only $\Dumb_1(3,3)$, shown in Figure~\ref{fig:C-extremal}(\ref{fig:C-D133}). It has eight maximal matchings of size $2$ and one maximal matching of size $3$. Hence
\[
 m=9,
 \qquad
 m'=19,
 \qquad
 \avm=\frac{19}{9}.
\]
\end{lemma}

\begin{proof}
Both cycles have length at least $3$ and the connector has length at least $1$, so the core is forced. The unique size-$3$ maximal matching uses the connector edge and one edge opposite the connector endpoint in each triangle. The other eight maximal matchings have size $2$.
\end{proof}

\section{Proof of the main theorem}\label{sec:global}

\begin{proof}[Proof of \cref{thm:global}]
For odd $n$, every matching has size at most $(n-1)/2$, and the complete equality list is the union of the three classwise lists.

Let $n=6$. The classwise maxima are
\[
 U_A(6)=\frac{13}{5},
 \qquad
 U_B(6)=\frac73,
 \qquad
 U_C(6)=\frac{19}{9}.
\]
Thus $U_A(6)$ is strictly largest, and \cref{thm:A-even} gives the unique extremal graph $\Th(1,3,3)$.

Now let $n\ge8$ be even. Set
\[
 U_A(n)=\frac n2-1+\frac{6}{3n-8},
 \quad
 U_B(n)=\frac n2-1+\frac2n,
 \quad
 U_C(n)=\frac n2-1+\frac2{n-4}.
\]
Then
\[
 U_C(n)-U_A(n)=\frac{8}{(n-4)(3n-8)}>0
\]
and
\[
 U_A(n)-U_B(n)=\frac{16}{n(3n-8)}>0.
\]
Hence $U_C(n)>U_A(n)>U_B(n)$, so the global extremal graphs are exactly the two $\Cclass$ families in \cref{thm:C-even}.
\end{proof}

\end{document}